\documentclass[a4paper,reqno,intlimits]{amsart}

\usepackage[T1]{fontenc}
\usepackage[latin1]{inputenc}
\usepackage[english]{babel}
\usepackage{amsfonts}
\usepackage{amssymb}
\usepackage{amsthm}
\usepackage{enumerate}
\usepackage{dsfont}
\usepackage{hyperref}

\newcommand{\calB}{\mathcal{B}}
\newcommand{\calD}{\mathcal{D}}
\newcommand{\calE}{\mathcal{E}}
\newcommand{\calH}{\mathcal{H}}
\newcommand{\calK}{\mathcal{K}}
\newcommand{\calL}{\mathcal{L}}
\newcommand{\calM}{\mathcal{M}}

\newcommand{\NN}{\mathds{N}}
\newcommand{\RR}{\mathds{R}}
\newcommand{\ZZ}{\mathds{Z}}

\DeclareMathOperator{\supp}{supp}
\DeclareMathOperator{\rmCap}{Cap}

\newcommand{\One}{\mathds{1}}
\newcommand{\rmd}{\mathrm{d}} 

\newtheorem{theorem}{Theorem}[section]
\newtheorem{corollary}[theorem]{Corollary}

\newtheorem{proposition}[theorem]{Proposition}

\theoremstyle{definition}
\newtheorem{definition}[theorem]{Definition}
\newtheorem{example}[theorem]{Example}

\theoremstyle{remark}
\newtheorem{remark}[theorem]{Remark}

\begin{document}

\title{Compactness of Schrödinger semigroups}

\author[D.~Lenz]{Daniel Lenz$^1$}
\author[P.~Stollmann]{Peter Stollmann$^2$}
\author[D.~Wingert]{Daniel Wingert$^3$}
\address{$^1$ Mathematisches Institut, Friedrich-Schiller Universit\"at Jena,
  Ernst-Abb\'{e} Platz 2, 07743 Jena, Germany}
\email{ daniel.lenz@uni-jena.de }
\address{$^2$ Fakult\"at f\"ur Mathematik,
           Technische Universit\"at, 09107 Chemnitz, Germany}
\email{ peter.stollmann@mathematik.tu-chemnitz.de }
\address{$^3$ Fakult\"at f\"ur Mathematik,
           Technische Universit\"at, 09107 Chemnitz, Germany }
\email{ daniel.wingert@s2000.tu-chemnitz.de }


\begin{abstract}
 This paper is concerned with emptyness of the essential spectrum, or equivalently compactness of the  semigroup, for perturbations of selfadjoint operators that are bounded below (on an $L^2$-space).

 For perturbations by a (nonnegative) potential we obtain a simple criterion for compactness of the semigroup in terms of relative compactness of the operators of multiplication with characteristic functions of sublevel sets. In the context of Dirichlet forms, we can even characterize compactness of the semigroup for measure perturbations. Here, certain 'averages' of the measure outside of  compact sets play a role.

 As an application we obtain compactness of semigroups for Schrödinger operators with potentials whose sublevel sets are thin at infinity.
 \bigskip
 
 This is the pre-peer reviewed version of the following article: \\
 Lenz, D., Stollmann, P., Wingert, D., Compactness of Schrödinger semigroups, Mathematische Nachrichten, which has been published in final form at \\
 \url{http://dx.doi.org/10.1002/mana.200910054}.
\end{abstract}

\maketitle

\section*{Introduction}

It is a classical fact, going back at least to Friedrichs \cite{Friedrichs-34} that a Schrödinger operator $-\Delta + V$ with a potential $V$ that goes to $\infty$ at $\infty$ has only discrete spectrum so that $\sigma_{ess}(-\Delta + V) = \emptyset$. This fact has attracted some renewed interest in recent years \cite{MazyaS-05, WangW-08, Simon-08} where the issue is first to come up with simple proofs and second to explore more general situations.
In this paper we add to this discussion with two main goals: first a rather easy method of proof and second a treatment of measure perturbations in the general Dirichlet form context.

To underline the simplicity, we start with a discussion of equivalent reformulations of the condition $\sigma_{ess}(H_0 + V) = \emptyset$ in terms of compactness of semigroups or, equivalently, resolvents. Our main results are Theorem \ref{thm:main1} and Theorem \ref{thm:equivessspec} below. In the first one, we present a useful notion of what it means that $V \rightarrow \infty$ at $\infty$ in operator theoretic terms (in the quantum setting if you wish): If we were to  talk about a measurable function on a locally compact space, $V \rightarrow \infty$ at $\infty$ would mean that the sublevel sets $\{ V \leq n \} := \{ x \in X | V(x) \leq n \}$ are relatively compact for all $n \in \NN$. The corresponding quantum condition is just that $\One_{\{ V \leq n \}}$ is relatively compact with respect to $H_0$ for all $n \in \NN$. This simple observation allows a particularly easy proof and gives a result that contains the above mentioned \cite{Simon-08,WangW-08}. While in the latter works the authors concentrated on the associated semigroups, we will see below that spectral projections or the inclusion map of the form domain come in handy. E.g., it is almost evident that additional negative perturbations can be allowed as long as they are form small. Still, mapping properties of the semigroup or the resolvent can be used to verify the assumption of relative compactness of the sublevel sets of the potential.

To give a satisfactory meaning to ''$\mu \rightarrow \infty$ at $\infty$'' is much harder for the case of a measure $\mu$. This situation is studied in some detail in Section \ref{sec:dirichletsetting} with Theorem \ref{thm:equivessspec} as the main result. An elegant criterion for the 1-d Laplacian, due to Molchanov, \cite{Molchanov-53}, says that
$-\Delta +\mu$ is compact if and only if $\mu(U+x)\to\infty$ as $x\to \pm\infty$ for some nonempty open interval (equivalently all nonempty open) $U\subset \RR$. It is quite easy to see that an analogous statement is wrong in dimensions $d\ge 2$. In the recent paper \cite{MazyaS-05}, Maz'ya and Shubin proved a compactness criterion in arbitrary dimension. Our result here goes back to the second named author's Habilitationsschrift \cite{Stollmann-94b} that gives a criterion for regular Dirichlet forms with ultracontractive semigroups, a setup that is much more general than the Laplacian in euclidean space.

Finally, we record the consequences of our main theorem for usual Schrödinger operators on Euclidean space in Section 4. Here the specific geometry gives a particularly nice sufficient condition for Schrödinger semigroups to be compact.

Compactness is one of the great concepts of analysis and it is a most comforting fact to learn that some operator is compact. But that is quite often hard to establish. In our  investigation below we take advantage of a smaller class of operators that is easier to deal with - the Hilbert-Schmidt operators - one of the great gifts of Erhard Schmidt, \cite{Schmidt-08}, to mankind.
\section{Relative spectral compactness and all that}

In this section, $\calH$ is a Hilbert space and $H$ some selfadjoint operator on $\calH$. The following notion is very useful in perturbation theory, see \cite{ReedS-78, Weidmann-80, Weidmann-00}; we need a rather easy special case, where the ``perturbation'' $B$ is bounded. We write $\calL = \calL(\calH)$ for the bounded operators and $\calK = \calK(\calH)$ for the ideal of compact operators, which is, of course a norm-closed subspace of $\calL$.

\begin{definition}
 An operator $B \in \calL(\calH)$ is called $H$-\emph{relatively compact} if $B \One_I(H) \in \calK(\calH)$ for every compact segment (bounded Borel set) $I \subset \RR$.
\end{definition}

Here, we use $\One_M$ to denote the indicator function of a set $M$ and, of course, the spectral theorem. As we will see below, instead of these indicator functions (or spectral projections) we could have used a variety of bounded functions of the operator.

\begin{proposition}\label{prop:gH_locspeccomp}
 Let $H$ be selfadjoint and $B \in \calL$.
 \begin{enumerate}[(1)]
  \item The following assertions are equivalent:
   \begin{enumerate}[(i)]
    \item $B$ is $H$-relatively compact.
    \item There is some $\varphi \in C(\sigma(H))$ with $| \varphi | > 0$ s.t. $B \varphi(H) \in \calK$.
    \item For all $\varphi \in C_0(\sigma(H)) : B \varphi(H) \in \calK$.
   \end{enumerate}
  \item Let $g \in C(\sigma(H), \RR)$ with $g(t) \rightarrow \infty$ as $| t | \rightarrow \infty$. Then $B$ is $g(H)$- relatively compact whenever $B$ is $H$-relatively compact.
 \end{enumerate}
\end{proposition}

\begin{proof} (1):  The implication (iii) $\Rightarrow$ (ii) is clear, as there obviously exist $\varphi \in C_0 (\sigma (H))$ with $\varphi >0$.

(ii) $\Rightarrow$  (i):  Just note that $\frac{1}{\varphi} \in C(\sigma(H))$ by assumption on $\varphi$ and, therefore, $(\One_I \frac{1}{\varphi})(H)$ is bounded. This gives
 \[ B \One_I(H) = B \varphi(H) (\One_I \frac{1}{\varphi})(H) \in \calK. \]

(i) $\Rightarrow$ (iii):  Note that by the functional calculus and since $\varphi \in C_0$,
 \[ \varphi(H) = \| \cdot \|-\lim_{n \rightarrow \infty} \varphi(H) \One_{[-n, n]}(H) \]
 \[ \Rightarrow B \varphi(H) = \| \cdot \|-\lim_{n \rightarrow \infty} B  \One_{[-n, n]}(H) \varphi(H) \in \calK. \]
 
 (2) For $I \subset \RR$ it is clear that $\One_I(g(H)) = \One_{g^{-1}(I)}(H)$ and by the assumption on $g$, the set $g^{-1}(I)$ is bounded.
\end{proof}

This gives easily the following result, where, of course, $D(H)$ denotes the domain of $H$ and is a Hilbert space endowed with the \emph{graph norm} $\| u \|_H^2 = \| u \|^2 + \| H u \|^2$. We also write $Q(H)$ for the form domain of $H$ in case $H \geq \gamma$, i.e. $H$ is semibounded below. Recall that in this case, $h[u, v] = (H u|v)$ is a closed form, when we take as a domain $D(h) = Q(H) = D((H+s)^{1/2})$ where $ s > -\gamma$ can be chosen arbitrarily. In analogy with the usual Sobolev spaces on $\RR^d$ and $H = -\Delta$, it is suggestive to write
\[ \calH^p(H) = D((H+s)^{p/2}), \]
so that $\calH^2(H) = D(H)$ and $\calH^1(H) = Q(H)$. Of course, these spaces are endowed with the respective graph norms and continuously embedded in $\calH$.

\begin{theorem}\label{cor:equiv_locspeccomp}
 The following are equivalent:
 \begin{enumerate}[(i)]
  \item $B$ is $H$-relatively compact,
  \item $B(H-\lambda)^{-k} \in \calK$ for some (all) $\lambda \in \rho(H), k \in \NN$,
  \item $B : D(H) \rightarrow \calH$ is compact.
 \end{enumerate}
 If $H \geq \gamma$ these conditions are in turn equivalent to each of the following:
 \begin{enumerate}[(i)] \setcounter{enumi}{3}
  \item $B e^{-t H}$ is compact for some (all) $t > 0$,
  \item $B:Q(H) \rightarrow \calH$ is compact,
  \item $B:\calH^p(H) \rightarrow \calH$ is compact for some (all) $p > 0$,
  \item for any $ C > 0$ and $(\psi_n) \subset Q(H)$ with $h[\psi_n] \leq C$ and $\psi_n \rightarrow 0$ weakly, it follows that $B\psi_n \rightarrow 0$ in norm.
 \end{enumerate}
\end{theorem}

\begin{remark}
 The equivalence (i) $\Leftrightarrow$ (ii) is essentially contained in \cite{Weidmann-00} Theorem 9.17.
\end{remark}

\begin{proof}
 The proof is an easy consequence of Proposition \ref{prop:gH_locspeccomp}. Here are the details.

 The equivalence between (i) and (ii) is immediate from part (1) of the proposition. Moreover, clearly, $B:D(H) \rightarrow \calH$ can be written as $B(H-\lambda)^{-1}(H-\lambda)$ for some $\lambda \in \rho(H)$, where the first factor is bounded from $(D(H), \| \cdot \|_H)$ to $\calH$. This gives (ii) $\Longrightarrow$ (iii). Conversely, $(H+\lambda)^{-1}$ is a bounded operator from $\calH$ to $D(H)$ and (iii)$\Longrightarrow $ (ii) follows.

 Note that $\varphi(x) = e^{-t x}$ belongs to $C_0 (\sigma (H))$ and the equivalence of (iv) to, say, (i) follows from (1) of the previous proposition.  The statements in (v) and (vi) are just statement (iii) with $H$ replaced by $g(H)$ for $g(t) = (t+s)^{p/2}$. As $g$ has an inverse function (which again tends to $\infty$ for $t\to \infty$), the equivalence of (v) and (vi) to (i) follows from part (2) of the previous proposition.

 Finally, (vii) is a simple reformulation of (v).
\end{proof}

\begin{corollary}\label{cor:equiv_emptyess}
 For $H$ selfadjoint, the following are equivalent:
 \begin{enumerate}[(i)]
  \item $Id$ is $H$-relatively compact,
  \item $\sigma_{ess}(H) = \emptyset$,
  \item $(H-\lambda)^{-1} \in \calK$ for some (all) $\lambda \in \rho(H)$.
 \end{enumerate}
 If $H \geq \gamma$ then these conditions are in turn equivalent to each of the following:
 \begin{enumerate}[(i)] \setcounter{enumi}{3}
  \item $e^{-t H} \in \calK$ for some (all) $t > 0$,
  \item $Id : Q(H) \rightarrow \calH$ is compact,
  \item $Id : \calH^p(H) \rightarrow \calH$ is compact for some (all) $p > 0$,
  \item for any $ C > 0$ and $(\psi_n) \subset Q(H)$ with $h[\psi_n] \leq C$ and $\psi_n \rightarrow 0$ weakly, it follows that $\psi_n \rightarrow 0$ in norm.
 \end{enumerate}
\end{corollary}

Of course this latter is basically well-known, see, e.g., \cite{ReedS-78}, Theorem XIII.64, p. 245. The equivalence of (i) and (vii) in the above corollary immediately gives:

\begin{example}[\cite{KulczyckiS-06} Theorem 1.1 part two]
 Let $H_0 \geq 0$ be a translation invariant, selfadjoint operator on $\RR^d$, $Q(H_0) \cap L^2(B_r(0)) \neq \emptyset$ for some $r > 0$ and $0 \leq V^+ \leq M < \infty$ on a sequence of disjoint balls with radius $r$. Then $\sigma_{ess}(H_0 + V^+) \neq \emptyset$.
\end{example}

We close this section with noting two simple stability results for emptyness of the essential spectrum.

\begin{corollary}
 Let $H$ be selfadjoint and nonnegative, $h$ the associated form and $\mu$ be a sesquilinear form.

 (a) If $\mu$ is form small with respect to $h$ i.e. form bounded with bound less than one, then
 \[ \sigma_{ess}(H) = \emptyset \Leftrightarrow \sigma_{ess}(H + \mu) = \emptyset. \]

 (b) If $\mu$ is nonnegative and  $h+\mu$ is closed, then $\sigma_{ess} (H+ \mu) =  \emptyset$ if $\sigma_{ess} (H) = \emptyset$.
\end{corollary}

\begin{proof} This follows easily by comparing unit balls in $Q(H)$ and $Q(H+\mu)$ and
 considering condition (v) in the previous corollary.
\end{proof}

\section{Schrödinger semigroups}

We will now assume $\calH = L^2(X)$, where $(X, \calB, m)$ is some measure space, and $H_0 \geq \gamma$.

We will study perturbations $H = H_0 + V$ where $V$ is a function on $X$ which is at the same time regarded as the operator of multiplication with this function. $H$ is defined via its quadratic form in the usual way: Assume that $V = V_+ - V_-$ where $V_+:X \rightarrow [0, \infty]$ is measurable and $V_-:X \rightarrow [0, \infty)$ is measurable. We first define the closed form of $H_0 + V_+$ as the form sum with form domain $Q(H_0 + V_+) = D(H_0^{1/2}) \cap D(V_+^{1/2})$, which might not be dense but that does not pose a problem. The associated selfadjoint operator is simply defined on the possibly smaller Hilbert space $\overline{Q(H_0+V_+)}$, the closure taken in $\calH$. For $V_-$ we require that it is form small w.r.t. $H_0 + V_+$, i.e., there are some $q < 1$ and $C_q \in \RR$ such that
\[ (V_- u \mid u) \le q \left( ( H_0 + V_+ ) u \mid u \right) + C_q \| u \|^2 \]
for all $u$. Then $H_0 + V_+ - V_-$ can be defined by the KLMN theorem (\cite{ReedS-75}, Theorem X.17) and we have $Q(H_0 + V_+ - V_-) = Q(H_0 + V_+)$.

The reader might have noticed that we didn't require $V_+, V_-$ to be the actual positive and negative parts of $V$ (thanks to Vitali Liskevich for pointing this out). Moreover, our assumption is obviously weaker than the usual assumption that $V_-$ is form small w.r.t. $H_0$. Here is our general theorem.

\begin{theorem}\label{thm:main1}
 Let $H_0, V$ and $H$ be as above. Assume that $\One_{\{V_+ < s\}}$ is $(H_0 + V_+)$-relatively-compact for some $s \in \RR$. Then $\sigma_{ess}(H) \subset [(1-q)(\gamma+s) - C_q, \infty)$.

In particular, if $\One_{\{V_+ \leq n\}}$ is $(H_0 + V_+)$-relatively compact for all $n \in \NN$, then $\sigma_{ess}(H) = \emptyset$, or, equivalently, $e^{-t H} \in \calK$ for all $t > 0$.
\end{theorem}

Theorem \ref{thm:main1} part 1 can be seen as a consequence of a well-known stability result about the essential spectrum under relatively compact perturbations, see \cite{Weidmann-00} Theorem 9.16.

\begin{proof}
 For $\lambda \in \sigma_{ess}(H)$ choose a Weyl sequence $f_n \rightarrow 0$ weakly, $\| f_n \| = 1$ and $\| H f_n - \lambda f_n \| \leq \frac{1}{n}$. Set $E := \{V_+ < s\}$. Since $(f_n)_{n\in\NN}$ is bounded with respect to the form norm, $\One_E f_n$ converges to zero by Theorem \ref{cor:equiv_locspeccomp}.
 We then have
 \begin{align*}
  \lambda &= \lim_{n\to\infty} \left((H_0+V_+-V_-)f_n\mid f_n\right) \\
   &= \lim_{n\to\infty}\left[\left( (H_0+V_+)f_n\mid f_n\right)- \left(V_-f_n\mid f_n\right)\right]
 \end{align*}
 and can therefore estimate
 \begin{align*}
  \lambda  &\ge \limsup \left[\left( (H_0+V_+)f_n\mid f_n\right) - \left( q  \left( (H_0+V_+)f_n\mid f_n\right)+C_q \| f_n\|^2\right)\right] \\
   &\ge \limsup \left[(1-q)\gamma +(1-q)\left( V_+(\One_E+\One_{X\setminus E})f_n\mid f_n\right) -C_q\right] \\
   &\ge \limsup \left[(1-q)\gamma +(1-q)\left( V_+\One_{X\setminus E}f_n\mid f_n\right) -C_q\right] \\
   &\ge \limsup \left[(1-q)\gamma +(1-q) s \left(\One_{X\setminus E}f_n\mid f_n\right) -C_q\right] \\
   &\ge (1-q)\gamma + (1-q)s-C_q.
 \end{align*}
 Here, we used convergence of $\One_E f_n$ to zero in the last step. The ``in particular'' assertion is now clear.
\end{proof}

\begin{remark} The reasoning given in the proof is quite flexible. It can easily be adopted to show e.g. the following statement: Let $H_0$ be as above, $\mu^+$ be a Hermitian, positive semidefinite, bilinear form such that $h_0 + \mu^+$ is closed, $\mu^-$ form small w.r.t. $h_0 + \mu^+$ and $H$ defined via the form sum. If for all $n \in \NN$ there exists an $M_n \in \calB$ with $\One_{M_n}$ being $(H_0 + \mu^+)$-relatively compact and $n \| \One_{M_n^c} u \|^2 \leq \mu^+(u, u) \;\forall u \in Q(H)$. Then $\sigma_{ess}(H) = \emptyset$.
\end{remark}

The important feature in the following corollary is that the assumption concerning relative compactness is phrased in terms of the operator $H_0$ and can so be checked easily in applications.

\begin{corollary}
 Let $H_0 \geq 0$, $V_+ \geq 0$ be as above, $V_-$ form small w.r.t. $g(H_0)+V_+$, where $g:[0, \infty) \rightarrow [0, \infty)$ satisfies $g(x) \rightarrow \infty$ as $x \rightarrow \infty$. Assume that $\One_{\{V_+ \leq n\}}$ is $H_0 $-relatively compact for all $n \in \NN$. Then $\sigma_{ess}(g(H_0) + V) = \emptyset$.
\end{corollary}

\begin{proof}
 Use Proposition \ref{prop:gH_locspeccomp} (2) to see that $\One_{\{V \leq n\}}$ is $g(H_0)$-relatively compact under the assumptions of the corollary. Then, the preceding theorem applies.
\end{proof}

\begin{remark}
 This gives a substantial generalization of Corollary 1.4 from \cite{WangW-08}, where $g$ was supposed to be a Bernstein function.
\end{remark}

\begin{definition}
 We say that $H_0$ is \emph{spatially locally compact} if
 \[ \One_E \One_I(H_0) \in \calK(L^2(X, \calB, m)) \]
 for every compact $I \subset \RR$ and every $E \in \calB$ with $m(E) < \infty$.
\end{definition}

Of course $H_0$ is spatially locally compact iff $\One_E$ is $H_0$-relatively compact for every $E \in \calB$ with $m(E) < \infty$. At the same time, there are many instances, where spatial local compactness is well established. Our main theorem gives the following immediate consequence.

\begin{corollary}\label{cor:spatial}
 Assume that $H_0$ is spatially locally compact, $V,H$ as above. Assume that $m(\{ V_+\le n\})<\infty$ for all $n\in\NN$. Then $\sigma_{ess}(H) = \emptyset$, or, equivalently, $e^{-t H} \in \calK$ for all $t > 0$.
\end{corollary}

\begin{remark}\label{rem}
 \begin{enumerate}[(1)]
  \item If  $e^{-t H_0} : L^2 \rightarrow L^\infty$ for some $t > 0$, then $H_0$ is spatially locally compact: In fact $\One_E e^{-t H_0}$ factors through $L^\infty$ and the little Grothendieck theorem gives that it is a Hilbert-Schmidt operator, in particular compact. See the discussion in \cite{Stollmann-94a}, or \cite{DemuthSSV-95} for the case of positivity preserving semigroups.

  Therefore, our Corollary \ref{cor:spatial} contains Theorem 2 from \cite{Simon-08} and Cor. 1.2. from \cite{WangW-08} as special cases. It seems that our proof is shorter and easier than Simon's, \cite{Simon-08}, which is in turn much more elementary than the proof of Wang and Wu \cite{WangW-08}. In \cite{Simon-08}, Theorem 2.2 there is some additional information on semigroups: For positive selfadjoint operators $A$ and $B$:
  \[ e^{-A}e^{-B}\in\calK \Longrightarrow e^{-(A+B)}\in\calK. \]
 But this can also be deduced along the lines above: By Theorem \ref{cor:equiv_locspeccomp} the assumption implies that $e^{-A}$ is $B$-relatively compact. Therefore, $\One_I(A)$ is $B$-relatively compact for any bounded $I\subset \RR$ and this gives the desired compactness.

 \item Note that the semigroups involved need not be positivity preserving, so $H_0$ may well come from some elliptic partial differential operator of higher order. Note also that $e^{-t(H_0 + V)}$ is not required to map $L^2$ ot $L^\infty$. (Thanks again to Vitali Liskevich.)
  \item For $X$ being Euclidean space or a manifold, the required spatial local compactness of $H_0$ is sometimes easily checked in terms of compactness of Sobolev embeddings, i.e. in variants of Rellich's theorem \cite{Kato-66}, Theorem V.4.4, see also the discussion in Section 4 below.
  \item The Laplacian on quantum or metric graphs is spatially locally compact under quite general assumptions, since $D(H_0)$ is continuously embedded in $L^\infty$, see \cite{LenzSS-08}.
  \item For combinatorial graphs, the condition of spatial local compactness is trivially satisfied, as $\One_E$ has finite rank in this case. Therefore we get a rather easy and not very subtle criterion in that case.
 \end{enumerate}
\end{remark}

We now turn to  providing an alternative short proof of a main result (Theorem 1.1) of \cite{WangW-08} within our approach, showing that our result is more general than the latter. The result requires the existence of a kernel as well as the validity of the inequality
\begin{equation}\label{eq:spi}
 \| f \|^2 \leq r \cdot h[f] + \beta(r) \| f \|_1^2\tag{$\star$}
\end{equation}
with some function $\beta$ defined on $[0, \infty)$, called the \emph{super Poincar\'{e} inequality}.

\begin{theorem} Let $H_0\ge 0$ be selfadjoint with an associated form $h$ that satisfies (\ref{eq:spi}). Let $V_+$ be as above and assume that $e^{-t(H_0+V_+)}$ is a substochastic operator with a kernel for some $t>0$. Then $H$ is spatially locally compact. In particular,  if $m(\{ V_+\le n\})<\infty$ for all $n\in\NN$ then $\sigma_{ess}(H) = \emptyset$, or, equivalently, $e^{-t H} \in \calK$ for all $t > 0$.
\end{theorem}

\begin{proof}
Let $E\subset X$ be a subset of finite measure, $T=e^{-t(H_0+V_+)}$.\\
 (1) By our assumption $T$ is substochastic and hence continuous from $L^\infty$ to $L^\infty$. Thus, its   kernel belongs to $L^1$ pointwise almost everywhere. \\
 (2) $T(A)$ is compact in $L^1$ for every $L^\infty$-bounded subset $A \subset L^1(E)$: Let $f_n \in A$ be a sequence. Without loss of generality, $f_n$ is weak-* convergent in $L^\infty$ (choose a proper subsequence). But then with (1) $T f_n$ converges pointwise a.e.
 By substochasticity of $T$ we have furthermore $|T f_n| \leq a \cdot T \One_E \in L^1$ for some $a \in \RR$. Therefore $T f_n$ converges and $T A$ is compact. \\
 (3) Let $B$ be the unit ball in $L^2$. Then $T \One_E B$ is compact in $L^1$. This follows directly with (2) and the fact that $\One_E B$ is uniformly integrable in $L^1$.\\
 (4) Assume $T \One_E B$ not to be compact in $L^2$. Then there exists  a sequence $f_n \in T \One_E B$ with
 \[ \| f_n - f_m \|^2 \geq \epsilon \]
 for some $\epsilon > 0$. Moreover, spectral calculus shows that  there exists an $C\geq 0$ such that
 \[ h[f - g] \leq C \]
 for all $f,g\in  T B$.  Now choose $r := \frac{\epsilon}{2c}$. Then $\frac{\epsilon}{2 \beta(r)} \leq \| f_n - f_m \|_1^2$ by (\ref{eq:spi}). This contradicts (3).
\end{proof}

We finish this section by showing that validity of (\ref{eq:spi}) is a direct consequence of ultracontractivity.

\begin{proposition} Let $H \geq 0$ be selfadjoint, $h$ the associated form and assume that the associated semigroup is ultracontractive. Then (\ref{eq:spi}) is valid.
\end{proposition}

\begin{proof} Spectral calculus and simple estimates show
$\| e^{-t H}f - f \|^2 \leq 2 t h[f, f].$
This yields
\[ \|f\|_2^2 \leq 2 t h [f,f] + \|e^{-t H} f\|_2. \]
As the semigroup is ultracontractive,  it maps $L^1$ continuously into $L^2$ and the statement follows.
\end{proof}

\section{Compactness of measure perturbations}\label{sec:dirichletsetting}

In this section we consider regular Dirichlet forms: so $X$ is assumed to be a locally compact, $\sigma$-compact metric space, $\calB$ the Borel-$\sigma$-field and $m$ a regular Borel measure with $\supp m = X$. We assume that $H_0 \geq 0$ is associated with a regular Dirichlet form $\calE$ with domain $\calD = \calD(\calE) = D(H_0^{1/2})$. By
\[ \rmCap(U) := \inf \{ \calE[\varphi, \varphi] + \| \varphi \|^2 \mid \varphi \geq \One_U, \varphi \in \calD \} \]
we define the \emph{capacity} of $U$, for $U$ open; one can then extend $\rmCap(\cdot)$ in the usual way to an outer regular setfunction by letting
\[ \rmCap(E) := \inf \{ \rmCap(U) \mid U \supset E, U \text{ open } \}; \]
see \cite{FukushimaOT-94} for details. From \cite{FukushimaOT-94}, we also infer that every $u \in \calD$ admits a \emph{quasi-continuous version} $\tilde{u}$, the latter being unique up to sets of capacity zero. This allows us to consider measure potentials in the following way: see \cite{Mazya-64}, \cite{Mazya-85} for the special case of the Laplacian and locally finite measures, \cite{Stollmann-92} and the references in there.

Let $\calM_0 = \{ \mu : \calB \rightarrow [0, \infty] \mid \mu \mbox{ a measure  }\mu \ll \rmCap \}$, where $\ll$ denotes absolute continuity, i.e. the property that $\mu(B) = 0$ whenever $B \in \calB$ and $\rmCap(B) = 0$. For measures in $\calM_0$ we explicitly allow that the measure takes the value $\infty$. A particular example is $\infty_B$, defined by
\[ \infty_B(E) = \infty \cdot \rmCap(B \cap E), \]
with the convention $\infty \cdot 0 = 0$. Note that for $\mu \in \calM_0$ we have that $\mu[u, v] := \int_X \tilde{u} \tilde{v} \: \rmd \mu$ is well defined for $u, v \in \calD(\mu)$, where $\calD(\mu) = \{ u \in \calD \mid \tilde{u} \in L^2(X, \calB, \mu) \}$. It is easy to see that
\[ \calD(\calE + \mu) := \calD \cap \calD(\mu), (\calE + \mu)[u, v] := \calE[u, v] + \mu[u, v] \]
gives a closed form (not necessarily densely defined). One can check that, e.g., $\calE + \infty_B = \calE |_{\calD_0(B^c)}$, where $\calD_0(U) = \{u \in \calD \mid \tilde{u} |_{U^c} = 0 \text{ q.e.} \}$.

For the Laplacian in $\RR^d$, $B$ closed (so $U$ is open) we get that $\calD(\calE + \infty_B) = W_0^{1, 2}(U)$ so that the form sum is $H_0 + \infty_B = - \Delta |_U$ with Dirichlet boundary condition.

If $\mu^- \in \calM_0$ is form small w.r.t. $\calE + \mu^+$, we can furthermore define $\calE + \mu = \calE + \mu^+ - \mu^-$ by the KLMN-theorem, already referred to above. Note that for $\mu^+ = 0$ this form boundedness implies that $\mu^-$ is a Radon measure, i.e., finite on all compact sets.

It is now an interesting question to determine what the property $V \rightarrow \infty$ at $\infty$ means for measures. For the classical Dirichlet form on $\RR^1$ this was answered by Molchanov \cite{Molchanov-53} who proved that $\sigma_{ess}(-\Delta + \mu) = \emptyset \Leftrightarrow \mu(U + x) \rightarrow \infty$ for every open interval $U$ and $x \rightarrow \infty$.

The direct analog is not appropriate in higher dimensions but a recent characterization can be found in \cite{MazyaS-05}. Here, we state and prove a characterization (originally from \cite{Stollmann-94b}) in the much more general setting of Dirichlet forms with ultracontractive semigroups. To this end, we introduce
\[ Av^\lambda_G(\mu) := \inf \{ \mu[u, u] \mid \| u \|_2 = 1, \calE[u] \leq \lambda, \tilde{u} = 0 \text{ q.e. on } G^c \} \]
for $G \subset X$, $\lambda \in \RR$.

\begin{theorem}\label{thm:equivessspec}
 Let $(K_n)_{n \in \NN}$ be a sequence of compact sets with $\bigcup K_n^\circ = X$, $G_n := K_n^c$. Assume that $H_0$ has ultracontractive semigroup, $\mu^+ \in \calM_0$ and $\mu^- \in \calM_0$ is form small with respect to $\calE + \mu^+$. Then the following are equivalent:
 \begin{enumerate}[(i)]
  \item $\sigma_{ess}(H_0 + \mu) = \emptyset$.
  \item For any $C \geq 0$ and $(f_n) \subset \calD$ with $(\calE + \mu + 1)[f_n] \leq C$ and $\tilde{f_n} = 0$ on $K_n$, it follows that $\| f_n \|_2 \rightarrow 0$ as $n \rightarrow \infty$.
  \item    $Av^\lambda_{G_n}(\mu) \rightarrow \infty \text{ as } n \rightarrow \infty$ for each $\lambda >0$.
 \end{enumerate}
\end{theorem}
\begin{proof}
 (i) $\Rightarrow$ (ii) By (v) of Corollary \ref{cor:equiv_emptyess} the sequence $(f_n)$ is relatively compact.
  Since $f_n = 0$ a.e. on $K_n$ the sequence furthermore  satisfies $f_n \rightarrow 0$ weakly.
 Consequently, $f_n\to 0$ in norm as well. \\
   (ii) $\Rightarrow$ (iii) Assume that $Av^\lambda_{G_n}(\mu) \nrightarrow \infty$; then we can obviously find a sequence that contradicts (ii). \\
 (iii) $\Rightarrow$ (i) Assume $\sigma_{ess}(H_0 + \mu) \neq \emptyset$. Then there must be a singular sequence $(g_n) \subset D(H_0 + \mu)$ with $\| g_n \| = 1$ and $\|(H_0 + \mu) g_n - \lambda g_n \| \leq \frac{1}{n}$ for all $n \in \NN$. In particular
 \[ (\calE + \mu)[g_n, g_n] \rightarrow \lambda. \]
 Since $K_l$ is compact we can infer from \cite{Stollmann-94a}, Corollary to Theorem 1, that $\sigma_{ess}(H_0 + \mu + \infty_{K_l}) = \sigma_{ess}(H_0 + \mu) \ni \lambda$; so that there is a sequence $(g_n^{(l)})_{n\in\NN} \subset \calD(\calE + \mu + \infty_{K_l})$ with the same properties. We choose $f_n = g_n^{(n)}$ and see that $Av_{G_n}^\eta(\mu) \nrightarrow \infty$, for appropriate $\eta$. In fact 
 \[ (\calE + \mu_+ + 1)[f_n] \leq c (\calE + \mu + 1)[f_n] + d \| f_n \|^2, \]
 since $\mu_-$ is form small w.r.t. $\calE + \mu_+$. This implies 
 \[ \calE[f_n] \leq \eta \text{ for } \eta = c (\lambda + 1) + d + 1 \]
 and $n$ large enough.
\end{proof}

\section{Back to Euclidean space}
Let us now consider the case of ordinary Schrödinger operators in $L^2(\RR^d)$. In this setting $H_0=-\Delta$ is spatially locally compact since the semigroup is ultracontractive, see Remark \ref{rem} above. Therefore, we know that $H_0+V_+$ has empty essential spectrum, whenever the nonnegative, measurable function $V_+:\RR^d\to [0,\infty]$ satisfies the condition that all sublevel sets $\{V_+\le n\}$ have finite measure, for $n\in\NN$. However, this latter condition is too strong, as was observed already by Rellich, \cite{Rellich-48}, a fact we learned from \cite{Simon-08}. In this latter paper, a sufficient condition can be found that covers Rellich's example. The following condition is obviously weaker:

\begin{definition}
 Denote by $C(k)$ the unit cube in $\RR^d$ centered at $k\in\RR^d$. We call a measurable set $E\subset \RR^d$ \emph{thin at infinity} if
 \[ \lim_{|k|\to\infty}|E\cap C(k)|=0. \]
\end{definition}

Using the Birman-Solomyak space
\[ \ell^\infty(L^2)=\{ f\in L^2_{loc}\mid \| f\|_{2;\infty}:= \sup_{k\in\ZZ^d}\|\One_{C(k)}f\|_2<\infty\} \]
and its closed subspace
\[ c_0(L^2)=\{ f\in L^2_{loc} \mid  \lim_{k\to\infty}\|\One_{C(k)}f\|_2=0\}, \]
we see that $E\subset \RR^d$ is thin at infinity iff $\One_E\in c_0(L^2)$.

\begin{theorem}\label{S.O.}
 Let $V_+,V_-\ge 0$ be measurable and assume that $V_-$ is form small with respect to $-\Delta$. If $\{V_+\le n\}$ is thin at infinity for every $n\in\NN$, then $\sigma_{ess}(-\Delta+V_+-V_-)=\emptyset$.
\end{theorem}

\begin{proof}
 For fixed $n\in\NN$ and $E=\{V_+\le n\}$ we show that
 \[ \One_E(-\Delta +1)^{-p}\in \calK \]
 for some $p\in\NN$. (Then, $\One_E (- \Delta + V_+  + 1)^{-1} \in \calK$ as well and the statement follows from Theorem \ref{thm:main1}.)

 An inequality of Strichartz, \cite{Strichartz-67} and \cite{Simon-05}, Lemma 4.10, gives that, for any $h\in\ell^\infty(L^2)$ and $p>\frac{d}{4}$:
 \[ \| h(-\Delta+1)^{-p}\| \le c\cdot \| h\|_{2;\infty}. \]
 Therefore, for $E$ as above, we get
 \[ \One_E(-\Delta +1)^{-p}=\|\cdot\|-\lim_{R\to\infty}\One_E\One_{B_R(0)}(-\Delta +1)^{-p}\in\calK, \]
 since $-\Delta$ is spatially locally compact.
\end{proof}

Playing with the equivalences from Theorem \ref{cor:equiv_locspeccomp} we get the following nice Corollary. It shows that $-\Delta$ could be replaced by  quite a variety of operators: roots of the Laplacian, or relativistic Laplacians, subelliptic operators, as well as elliptic partial differential operators:

\begin{corollary}
 Let $H_0\ge \gamma$ be selfadjoint with $Q(H_0)$ continuously embedded in $W^{s,2}(\RR^d)$ for some $s>0$  
 Let $V_+,V_-\ge 0$ be measurable and assume that $V_-$ is form small with respect to $H_0$. If $\{V_+\le n\}$ is thin at infinity for every $n\in\NN$, then $\sigma_{ess}(H_0+V_+-V_-)=\emptyset$.
\end{corollary}

\begin{proof}
 From Theorem \ref{S.O.} and Theorem \ref{cor:equiv_locspeccomp} we deduce that $\One_E:  W^{s,2}(\RR^d)\to L^2$ is compact, since $ W^{s,2}(\RR^d)=\calH^{\frac{s}{2}}(-\Delta)$ in the notation introduced in Section 1.
\end{proof}

The referee kindly pointed references \cite{BenciF-78a, BenciF-78b} that contain criteria for emptyness of the essential spectrum of Schrödinger operators. The condition in these papers is $V = V_1 + V_2$,
\[ \int_{B(x,1)}(V_1+C)^{-1}(y)dy \to 0 \mbox{  for  }| y|\to\infty ,\qquad(\star \star ) \]
where $C$ is a suitable lower bound for $V_1$ and $B(x,1)$ denotes the unit ball shifted to $x$. It is easy to see that $(\star \star )$ implies that $\{V_1\le n\}$ is thin at infinity for every $n\in\NN$. In the first paper, $V_2=0$ and $C=0$ is assumed and in the second paper, negative perturbations $V_2$ in $L^\frac{d}{2}$ are allowed for $d\ge 2$. Clearly, our Theorem \ref{S.O.} is more general.

\subsection*{\textbf{Acknowledgment}}
P.S. thanks Barry Simon for a fruitful correspondence on the subject matter as well as Vitali Liskevich for some useful hints. Thanks go to a referee for several helpful remarks.

\def\cprime{$'$}
\providecommand{\bysame}{\leavevmode\hbox to3em{\hrulefill}\thinspace}

\end{document}